\documentclass[draft]{amsart}

\usepackage{amsthm}
\usepackage{mathtools}
\usepackage{amssymb} 
\usepackage{enumerate} 

\usepackage{hyperref}
\usepackage{tikz-cd}
\numberwithin{equation}{section}

\theoremstyle{plain}
\newtheorem{lemma}{Lemma}[section]
\newtheorem{theorem}[lemma]{Theorem}
\newtheorem{proposition}[lemma]{Proposition}
\newtheorem{corollary}[lemma]{Corollary}
\newtheorem{question}[lemma]{Question}

\theoremstyle{definition} 
\newtheorem{definition}[lemma]{Definition}
\newtheorem{example}[lemma]{Example}
\newtheorem{notation}[lemma]{Notation}

\newtheorem{remark}[lemma]{Remark} 

\renewcommand{\dim}{\operatorname{dim}}
\newcommand{\level}{\operatorname{level}}

\newcommand{\Ext}{\operatorname{Ext}}

\newcommand{\hh}{\operatorname{H}} 
\newcommand{\zz}{\operatorname{Z}}
\newcommand{\bb}{\operatorname{B}}
\newcommand{\cc}{\operatorname{C}}

\newcommand{\Hom}{\operatorname{Hom}}
\newcommand{\susp}{\mathsf{\Sigma}}

\newcommand{\im}{\operatorname{im}}
\newcommand{\coker}{\operatorname{coker}}

\renewcommand{\le}{\leqslant}
\renewcommand{\ge}{\geqslant}

\newcommand{\pd}{\operatorname{pd}}
\newcommand{\Gpd}{\operatorname{Gpd}}

\newcommand{\RHom}{\operatorname{\mathsf{R}Hom}}

\newcommand{\sfC}{\mathsf C} 
\newcommand{\sfD}{\mathsf D} 
 
\newcommand{\sfG}{\mathsf G}
\newcommand{\sfT}{\mathsf T}
\newcommand{\sfS}{\mathsf S}
\newcommand{\sfA}{\mathsf A}
\newcommand{\sfB}{\mathsf B}
\newcommand{\sfP}{\mathsf P}
\newcommand{\thick}{\operatorname{thick}}
\newcommand{\add}{\operatorname{add}}
\newcommand{\smd}{\operatorname{smd}}

\begin{document}

\title[Level and Gorenstein projective dimension]{Level and Gorenstein projective dimension}

\author[L. Awadalla]{Laila Awadalla}

\address{University of Nebraska-Lincoln, Lincoln, NE 68588, U.S.A.}
\email{laila.awadalla@huskers.unl.edu}

\author[T.\ Marley]{Thomas Marley}

\address{University of Nebraska-Lincoln, Lincoln, NE 68588, U.S.A.}
\email{tmarley1@unl.edu}

\urladdr{http://www.math.unl.edu/tom-marley}

\date{\today}

\bibliographystyle{amsplain}



\begin{abstract}  We investigate the relationship between the  level of  a bounded complex over  a commutative ring with respect to the class of Gorenstein projective modules and other invariants of the complex or ring, such as projective dimension, Gorenstein projective dimension, and Krull dimension.  The results build upon work done by J.\ Christensen \cite{JChr},  H.\ Altmann et al. \cite{AGMSV},  and Avramov et al. \cite{ABIM} for levels with respect to the class of finitely generated projective modules. \end{abstract}

\maketitle

\section{Introduction}  The concept of {\it level} in a triangulated category, first defined by Avramov, Buchweitz, Iyengar, and Miller \cite{ABIM}, is a measure of how many mapping cones (equivalently, extensions) are needed to build an object from a collection of other objects, up to suspensions, finite sums, and retractions.  This concept has its origins in the works of Beilinson, Bernstein, and Deligne \cite{BBD}, J.\ Christensen \cite{JChr},  Bondal and Van den Bergh \cite{BV}, Rouquier \cite{R}, and others.    In particular, the concept of level is implicit in Rouquier's definition of dimension of a triangulated category.  

In the case of the bounded derived category of a commutative Noetherian local ring, levels have been used to establish, among other things, a lower bound on the sum of the Loewy lengths of the homology modules of (non-acyclic) perfect complexes (\cite[Theorem 10.1]{ABIM}).  In this context, it is interesting to compare the level of an object  with other more familiar homological invariants.  
For instance, the level of a finitely generated module (considered as a complex concentrated in degree zero) with respect to the ring is one more than the projective dimension of the module (\cite{JChr}; see also \cite[Cor. 2.2]{AGMSV}).   On the other hand, since the level of an object and its suspension are the same, uniform bounds on levels may exist in situations where there are no such bounds for homological dimensions.  For example, a local ring is regular if and only if the level with respect to the ring of any bounded complex with finite homology is at most one more than the dimension of the ring \cite[Theorem 5.5]{ABIM}, whereas the projective dimensions of such complexes over a regular local ring may be arbitrarily large.

In this paper, we study the levels of complexes with respect to the class ${\sfG}$ of Gorenstein projective modules in the bounded derived category of a commutative ring $R$.  
It is straightforward to prove that if $M$ is a bounded complex (by which we always mean homologically bounded) which is not acyclic then 
$$\level_R^{\sfG}M\le \Gpd_RM-\inf M +1,$$
where $\Gpd_RM$ is the Gorenstein projective dimension of $M$ and $\inf M=\inf \{n\mid \hh_n(M)\neq 0\}$ (Corollary \ref{easy-inequality}).  Typically, lower bounds for the level of an object are much harder to obtain.  One of our main results (Theorem \ref{main-theorem}) is the following:

\begin{theorem} \label{theoremA} Let $M$ be a bounded below complex which is not acyclic.  Then $$\level_R^{\sfG} M\ge \Gpd_R M - \sup M+1.$$
\end{theorem}

As an immediate consequence, we obtain the following generalization of \cite[Proposition 4.5]{JChr} and \cite[Cor 2.2]{AGMSV}:

\begin{corollary} \label{CorA} Let $M$ be a nonzero $R$-module of finite Gorenstein projective dimension.   Then 
$$\level_R^{\sfG}M=\Gpd_RM+1.$$
\end{corollary}

As with the proofs of \cite[Proposition 4.5]{JChr} and \cite[Cor 2.2]{AGMSV}, our proof of Theorem \ref{theoremA} relies critically on an application of the ``Ghost Lemma'' (cf. \cite[Theorem 3]{Ke}).  However, the argument here is significantly more complicated, as maps between (hard) truncations of Gorenstein projective resolutions are not necessarily $\sfG$-ghost.  We are able to work around this difficulty using a result of L.\ Christensen and Iyengar  (\cite[Theorem 3.1]{CI}) demonstrating the existence of Gorenstein projective resolutions of a specific form, along with the aid of an elementary  ``splitting lemma'' (Lemma \ref{tricky-lemma}).
As a consequence, we are able to prove the following characterization of Gorenstein local rings (Corollary \ref{Gor-ring}):

\begin{corollary} Let $R$ be a local Noetherian ring with residue field $k$. The following conditions are equivalent:
\begin{enumerate}[(a)]
\item $R$ is Gorenstein;
\item $\level_R^{\sfG} k<\infty$;
\item $\level_R^{\sfG} k=\dim R+1$;
\item $\level_R^{\sfG}M\le \dim R +\sup M-\inf M+1$ for all non-acyclic bounded below complexes $M$.
\end{enumerate}
\end{corollary}

Finally, as mentioned above, it is known that a commutative Noetherian local ring $R$ is regular if and only if $\level_R^R M\le \dim R+1$ for every bounded complex $M$ over $R$ with finitely generated homology (\cite[Theorem 5.5]{ABIM}).  We show that a direct analogue of this result for Gorenstein rings and $\level_R^{\sfG}M$ in place of $\level^R_R M$ does not hold (Example \ref{eg}).  However, we are able to establish a global bound for arbitrary Gorenstein local rings (Theorem \ref{global-bound}), although we do not know if the bound is best possible:

\begin{theorem} \label{global-bound} Let $R$ be a Gorenstein local ring of dimension $d$ and $M$ a complex in $\sfD_b(R)$.  Then 
$$\level_R^{\sfG} M\le 2d+2.$$
\end{theorem}
\smallskip
\noindent {\bf Acknowledgement:}  The authors would like to thank Lars Christensen and Sri Iyengar for helpful discussions in the course of working on this project.

\section{Preliminaries}

Throughout this paper $R$ will denote a commutative ring with identity.

\begin{subsection}{Complexes and derived categories} We will work with complexes over $R$, which we grade homologically:

$$M:= \cdots \to M_{n+1}\xrightarrow{\partial_{n+1}} M_n\xrightarrow{\partial_n} M_{n-1} \to \cdots $$

For each $n$ we let $\zz_n(X):=\ker \partial_n$, $\bb_n(X):=\im \partial_{n+1}$, $\cc_n(X):=\coker \partial_{n+1}$, and $\hh_n(M)=\zz_n(M)/\bb_n(M)$.
We also set $\sup M:=\sup \{n\mid \hh_n(M)\neq 0\}$ and $\inf M:=\inf \{n\mid \hh_n(M)\neq 0\}$.

We use the notation $\sfD(R)$ to denote the derived category of $R$.  Similarly, we let $\sfD_+(R)$ (respectively, $\sfD_b(R)$) denote the full subcategory of $\sfD(R)$ consisting of all $R$-complexes $M$ such that $\inf M>-\infty$, (respectively, $\sup M<\infty$ and $\inf M>-\infty$).  We let $\sfD^f(R)$ denote the full subcategory of $\sfD(R)$ consisting of all complexes whose homology is finitely generated in each degree. The subcategories $\sfD^f_+(R)$ and $\sfD_b^f(R)$ are defined similarly.  We use the symbol $\simeq$ to denote isomorphism in derived categories.   For any $R$-complex $M$ we let $M^{\#}$ denote the $R$-complex which is equal  to $M$ as a graded $R$-module  but whose  differentials are all zero.  We refer the reader to \cite{AF} for any unexplained terminology or notation regarding complexes.

\end{subsection}

\begin{subsection}{Gorenstein projective dimension}  In this subsection we summarize the basic properties of Gorenstein projective modules and Gorenstein projective dimension for complexes.  We refer the reader to \cite{Chr} and \cite{CFH} for proofs and additional detail than what is given here.

\begin{definition} \label{Gpd-defn} A complex $P$ of projective $R$-modules is called {\it totally acyclic} if it is acyclic and $\Hom_R(P,L)$ is also acyclic for every projective $R$-module $L$.    An $R$-module $G$ is called {\it Gorenstein projective} if $G$ is isomorphic to the cokernel of some differential of a totally acyclic complex.   \end{definition}

\begin{example}  Any projective $R$-module $P$ is Gorenstein projective as the complex $0\to P\xrightarrow{1} P\to 0$ is totally acyclic.
\end{example}

\begin{example}  If $R$ is a zero-dimensional Gorenstein local ring then every $R$-module is Gorenstein projective.  This follows easily from the fact that the class of projective $R$-modules and the class of injective $R$-modules are the same.  
\end{example}

We will need the following result:

\begin{lemma} \label{vanishing} Suppose $G$ is a Gorenstein projective $R$-module.  Then $\Ext^i_R(G,M)=0$ for all $i>0$ and all $R$-modules $M$ such that $\pd_RM<\infty$.
\end{lemma}
\begin{proof}  See \cite[Lemma 2.1]{CFH}.
\end{proof}

\begin{definition}(\cite[1.7]{CFH}) \label{Gpd} Let $M$ be a complex in $\sfD_+(R)$.  Let $\sfA$ be the class of all $R$-complexes $X$ such that $X^{\#}$ is bounded below and $X_i$ is Gorenstein projective for all $i$. The {\it Gorenstein projective dimension} of $M$  is defined by 
$$\Gpd_R M:= \inf \{ \sup X^{\#} \mid  X\in \sfA \text{ and }X \simeq M \text{ in }\sfD_+(R) \}.$$
\end{definition}

\begin{proposition} \label{Gpd-prop}  Let $M$ be a complex in $\sfD_+(R)$.  The following are equivalent:
\begin{enumerate}[(a)]
\item $\Gpd_R M\le n.$
\item $n\ge \sup M$ and for any (equivalently, some) $R$-complex $X\in \sfA$ with $X\simeq M$,  the module $\cc_n(X)$ is Gorenstein projective.
\end{enumerate}
\end{proposition}
\begin{proof}
See \cite[Theorem 3.1]{CFH}.
\end{proof}

\begin{proposition} \label{2-of-3} Let $X\to Y\to Z\to \susp X$ be an exact triangle in $\sfD_+(R)$.  If any two of $X$, $Y$, and $Z$ have finite Gorenstein projective dimension then so does the third.
\end{proposition}
\begin{proof}   This result follows from \cite[Theorem 3.9]{V}, since Gorenstein projective dimension is preserved under isomorphism in $\sfD_+(R)$ and any exact triangle is isomorphic to one induced by a short exact sequence of complexes.  (Also note that by Proposition \ref{Gpd-prop} and  \cite[Theorem 3.4]{V}, the definition of Gorenstein projective dimension given in \cite{V} agrees with Definition \ref{Gpd}.)
\end{proof}

\begin{theorem} \label{Gor-classic}  Let $R$ be a Noetherian local ring with residue field $k$.  The following are equivalent:
\begin{enumerate}[(a)]
\item $R$ is Gorenstein;
\item $\Gpd_R k<\infty$;
\item $\Gpd_R k=\dim R$;
\item $\Gpd_R M\le \sup M +\dim R$ for all nonzero complexes $M$ in $\sfD_+(R)$.
\end{enumerate}
\end{theorem}
\begin{proof} The result follows from  \cite[4.4.5 and 1.4.9]{Chr} and \cite[Proposition 3.8]{CFH}.
\end{proof}

We will need the following result concerning complexes of finite Gorenstein projective dimension.

\begin{theorem} \label{CI-thm}
Let $M$ be a complex in $\sfD_+(R)$ of finite Gorenstein projective dimension and $n$ an integer such that $\inf M\le n\le \Gpd_RM$.  Then $M\simeq X$ for some complex $X$ in $\sfD_+(R)$ such that 
\begin{equation}\notag
X_i \text{ is }
\begin{cases}
\text{zero},& \text{ if } i>\Gpd_RM \text{ or } i<\inf M \\
\text{projective},& \text{ if } i\ne	 n;\\
\text{Gorenstein projective},& \text{ if } i=n.
\end{cases}
\end{equation}
Moreover, if $R$ is Noetherian and $M$ is in $\sfD^f_+(R)$ then each $X_i$ may be chosen to be finitely generated.
\end{theorem}
\begin{proof} The result follows from the construction given in the proof of  \cite[Theorem 3.1]{CI}. See also \cite[Remark 3.5]{CI}.
\end{proof}

\end{subsection}

\begin{subsection}{Levels in triangulated categories}  We adopt the notation and terminology of Section 2 of  \cite{ABIM} regarding levels in a triangulated category.   

Let $\sfT$ be a triangulated category.   A subcategory of $\sfT$ is called {\it strict} if it is closed under isomorphisms in $\sfT$.  A triangulated subcategory of $\sfT$ is called {\it thick} if it is strict and closed under direct summands.   Equivalently, a subcategory $\sfS$ of $\sfT$ is thick if it is full, strict, closed under direct summands, and in any exact triangle $L\to M\to N\to \susp L$, if two of $L$, $M$, $N$ are in $\sfS$, so is the third.  It is readily seen that the intersection of thick subcategories is again thick.  As examples note that $\sfD_+(R)$ is a thick subcategory of $\sfD(R)$, and $\sfD^f_+(R)$ is a thick subcategory of $\sfD_+(R)$.   

Let $\sfC$ be a nonempty collection of objects of $\sfT$.   The {\it thick closure} of $\sfC$, denoted $\thick_{\sfT}(\sfC)$, is defined to be the intersection of all thick subcategories of $\sfT$ containing $\sfC$. 
For example, $\thick_{\sfD(R)}(R)$ is the subcategory of $\sfD(R)$ consisting of the perfect complexes, i.e., complexes which are isomorphic in $\sfD(R)$ to a bounded complex of finitely generated projective $R$-modules.

We let $\add(\sfC)$ (respectively, $\add^{\susp}(\sfC)$) denote the the intersection of all strict and full subcategories of $\sfT$ which contain $\sfC$ and are closed under finite sums (respectively, closed under finite sums and suspensions).   We let $\smd(\sfC)$ denote the intersection of all strict and full subcategories of $\sfT$ which contain $\sfC$ and are closed under direct summands.    

Let $\sfA$ and $\sfB$ be strict and full subcategories of $\sfT$.  We define $\sfA\star \sfB$ to be the full subcategory of $\sfT$ whose objects consist of all objects $M$ of $\sfT$ such that there exists an exact triangle $L\to M\to N\to \susp L$ where $L\in \sfA$ and $N\in \sfB$.   Evidently, $\sfA\star \sfB$ is also strict.  

For a collection of objects $\sfC$ of $T$ and a nonnegative integer $n$, we define the  {\it $n$th thickening} of $\sfC$ in $\sfT$ to be the full subcategory of $\sfT$ whose objects are defined as follows:  

$$
\begin{aligned}
\thick_{\sfT}^0(\sfC)&=\{0\};\\
\thick_{\sfT}^1(\sfC)&=\smd(\add^{\susp}(\sfC));\\
\thick_{\sfT}^n(\sfC)&=\smd(\thick_{\sfT}^{n-1}(\sfC)\star \thick_{\sfT}^1(\sfC)) \text{ for $n\ge 2$}.
\end{aligned}
$$
It is straightforward to show that $\thick^n_{\sfT}(\sfC)\subseteq \thick^{n+1}_{\sfT}(\sfC)$ for all $n\ge 0$ and that $$\thick_{\sfT}(\sfC)=\bigcup_{n\ge 0} \thick^n_{\sfT}(\sfC).$$
For an object $M$ of $\sfT$  we define the $\sfC$-{\it level} of $M$ in $\sfT$ by
$$\level_{\sf T}^{\sfC}M:=\inf \{n\ge 0\mid M\in \thick^{n}_{\sfT}(\sfC)\}.$$
Note that $\level_{\sf T}^{\sfC}(M)<\infty$ if and only if $M\in \thick_{\sfT}(\sfC)$.

We list a few basic facts regarding levels:

\begin{proposition}\label{level-facts} {\rm (\cite[Lemma 2.4]{ABIM})} Let $\sfT$ be a triangulated category and $\sfC$ a nonempty collection of objects of $\sfT$.   Let $L$, $M$ and $N$ be objects of $\sfT$.
\begin{enumerate}
\item $\level_{\sf T}^{\sfC}M=\level_{\sf T}^{\sfC}N$ if $M$ is isomorphic to $N$.
\item $\level_{\sfT}^{\sfC}M=\level_{\sfT}^{\sfC}(\susp^s M)$ for all integers $s$.
\item $\level_{\sfT}^{\sfC}M=\level_{\sfT}^{{\sfD}}M$ where $\sfD=\smd(\add(\sfC))$.
\item $\level_{\sfT}^{\sfC}M=\level_{\sfS}^{\sfC}M$ for any thick subcategory $\sfS$ of $\sfT$ containing $\sfC$.
\item $\level_{\sfT}^{\sfC}V\le \level_{\sfT}^{\sfC}U+\level_{\sfT}^{\sfC}W$ whenever $U\to V\to W\to \susp U$ is an exact triangle in $\sfT$.
\item $\level_{\sfT}^{\sfC}(M\oplus N)=\max\{\level_{\sfT}^{\sfC}M, \level_{\sfT}^{\sfC}N\}$.
\end{enumerate}

\end{proposition}

\begin{proof} Parts (1) and (2) follow from the fact that $\thick^n_{\sfT}(\sfC)$ is closed under isomorphisms and suspensions for all $n$.  Part (3) follows from $\thick^1_{\sfT}(\sfC)=\thick^1_{\sfT}(\sfD)$.  For (4), we note that for any collection of objects $\sfB$ of $\sfT$ which are contained in $\sfS$, $\add^{\susp}(\sfB)$ and $\smd(\sfB)$ are contained in $\sfS$.  Hence, $\thick^n_{\sfT}(\sfC)=\thick^n_{\sfS}(\sfC)$ for all $n$.

For part (5), suppose $\level_{\sfT}^{\sfC}U=m$ and $\level_{\sfT}^{\sfC}W=n$.   Then $U\in \thick^m_{\sfT}(\sfC)$ and $W\in \thick^n_{\sfT}(\sfC)$.  Thus, $V\in \thick^m_{\sfT}(\sfC)\star \thick^n_{\sfT}(\sfC)$.
By the definition of $\thick^m_{\sfT}(\sfC)$ and Lemma 2.2.1 of \cite{BV} we have that $\thick^m_{\sfT}(\sfC)=\smd(\add^{\susp}(\sfC)\star \cdots \star \add^{\susp}(\sfC))$ ($m$ factors).  Another application of this lemma gives that
$\thick^m_{\sfT}(\sfC) \star \thick^n_{\sfT}(\sfC)\subseteq \smd(\thick^m_{\sfT}(\sfC) \star \thick^n_{\sfT}(\sfC))= \thick^{m+n}_{\sfT}(\sfC)$.  Hence, $\level_{\sfT}^{\sfC}V\le m+n$.

For part (6), let $m=\level_{\sfT}^{\sfC}M$, $n=\level_{\sfT}^{\sfC}N$, and $\ell=\level_{\sfT}^{\sfC}(M\oplus N)$.  Since thickenings are closed under direct summands,  $m$ and $n$ are both at most $\ell$. This proves one inequality.  If either $m$ or $n$ is infinite, the reverse inequality is trivial.  Otherwise, suppose $m$ and $n$ are both finite and assume $m\le n$.  Then $M$ and $N$ are both in $\thick^n_{\sfT}({\sfC})$ and hence $M\oplus N\in \thick^n_{\sfT}({\sfC})$.  Thus, $\ell\le n=\max\{m,n\}$.
\end{proof}

The next result follows readily from Proposition \ref{level-facts}:

\begin{corollary} \label{level-inequality} Let $\sfC$ be a collection of objects of $\sfD(R)$ and $M$ an $R$-complex. Then

$$\level_R^{\sfC}M\le \inf \left\{\sum_{i\in \mathbb Z} \level_R^{\sfC}L_i\mid L\simeq M \text{ in }\sfD(R) \right\}.$$

In particular, if $M\simeq L$ in $\sfD(R)$ and $L_i\in \thick^1_{\sfD(R)}(C)$ for all $i$, then

$$\level_R^{\sfC}M\le \sup L^{\#}- \inf L^{\#}+1.$$

\end{corollary}

\begin{proof}  The first inequality follows by applying part (5) of Lemma \ref{level-facts} to the exact triangles $L_n\to L_{\ge n}\to L_{\ge n+1}\to \susp L_n$ for each $n$.  The second follows from the first and observing that $\level_R^{\sfC}L_i\le 1$ for all $i$.

\end{proof}

\begin{notation} Let $\sfC$ be a collection of objects in $\sfD(R)$ and $M$ an $R$-complex.  We let $\level_R^{\sfC}M$ denote $\level_{\sfD(R)}^{\sfC}M$.  Then by part (4) of Proposition \ref{level-facts}, $\level_R^{\sfC}M=\level_{\sfT}^{\sfC}M$ for any thick subcategory $\sfT$ of $\sfD(R)$ containing $\sfC$.   In the case $\sfC$ consists of a single object, say $\sfC=\{A\}$, we denote $\level_R^{\sfC}M$ by $\level_R^AM$.  Note that $\level_R^R M=\level_R^{\widetilde{\sfP}} M$, where $\widetilde{\sfP}$ is the class of finitely generated projective modules, since $\widetilde{\sfP}=\smd(\add R)$ and by Proposition \ref{level-facts}(3). We'll reserve the symbol $\sfP$ to denote the class of all projective $R$-modules.  \end{notation}

The following two results follow easily from the definitions and Corollary \ref{level-inequality}:

\begin{corollary} \label{perfect}  The following hold for any $R$-complex $M$:\begin{enumerate}[(a)]
\item $\level_R^{\sfP}M\le 1$ if and only if $M$ is isomorphic in $\sfD(R)$ to a bounded complex of projective $R$-modules with zero differentials;
\item $\level_R^RM<\infty$ if and only if $M$ is a perfect complex, i.e., $M$ is isomorphic in $\sfD(R)$ to a bounded complex of finitely generated projective modules.
\end{enumerate}
\end{corollary}

\begin{corollary} \label{level-pd} Let $M$ be a nonzero complex in $\sfD_+(R)$.  Then
$$\level_R^{\sfP}M\le \pd_R M-\inf M+1.$$
Moreover, if $R$ is Noetherian and $M$ is in $\sfD^f_+(R)$ then
$$\level_R^RM\le \pd_R M-\inf M+1.$$
\end{corollary}
\begin{proof} If $\pd_RM=\infty$ then both inequalities are trivial. Suppose $\pd_RM<\infty$ and let $P$ be a semi-projective replacement for $M$ such that $\sup P^{\#}=\pd_RM$.  We may furthermore assume $\inf P^{\#}= \inf M\ge \inf M^{\#}$.  Then by Corollary \ref{level-inequality} we have
$$
\begin{aligned}
\level_R^{\sfP}M&\le \sup P^{\#}-\inf P^{\#}+1\\
&\le \pd_RM-\inf M+1.
\end{aligned}
$$
For the second inequality, note that if $R$ is Noetherian and $H_n(M)$ is finitely generated for all $n$, we may assume the semi-projective replacement $P$ is finitely generated in each degree.  Hence, $\level_R^R(P_n)\le 1$ for all $n$.  The inequality again follows by Corollary \ref{level-inequality}.

\end{proof}

Essentially identical arguments yield analogous results for levels with respect to Gorenstein projective modules.  We let $\sfG$ (respectively, $\widetilde{\sfG}$) denote the class of all Gorenstein projective modules (respectively, finitely generated Gorenstein projective modules). 

\begin{corollary} \label{easy-inequality} Let $M$ be a nonzero complex in $\sfD_+(R)$. Then $$\level_R^{\sfG}M\le \Gpd_R M-\inf M+1.$$  Moreover, if $R$ is Noetherian and $M$ is a complex in $\sfD^f_+(R)$ then
$$\level_R^{\widetilde{\sfG}} M\le \Gpd_R M -\inf M +1.$$
\end{corollary}

We also have an analogous result to Corollary \ref{perfect}:

\begin{proposition} \label{finite-level}  For any $R$-complex $M$  we have:
\begin{enumerate}[(a)]
\item $\level_R^{\sfG}(M)\le 1$ if and only if $M$ is isomorphic in $\sfD(R)$ to a bounded complex of Gorenstein projective modules with zero differentials;
\item $\level_R^{\sfG}(M)<\infty$ if and only if $M$ is in $\sfD_+(R)$ and  $\Gpd_R M<\infty$.
\end{enumerate}

\end{proposition}
\begin{proof}  Part (a) is clear, as $\sfG$ as closed under finite sums and direct summands as well as isomorphisms.   For part (b), one implication follows from Corollary \ref{easy-inequality}.   For the reverse implication, we use induction on $n=\level_R^{\sfG}(M)$, with the case $n\le 1$ following from part (a).   Suppose $n>1$.  Then there exist a complex $Y$ and an exact triangle $X\to Y\to Z\to \susp X$ such that $M$ is (isomorphic to) a direct summand of $Y$,  $\level_R^{\sfG} X\le n-1$, and $\level_R^{\sfG}(Z)\le 1$.  By the induction hypothesis, $X$ and $Z$ are in $\sfD_+(R)$ and have finite Gorenstein projective dimension.  The same then holds for $Y$ (and hence $M$) by Proposition \ref{2-of-3}.
\end{proof}

\begin{remark} We remark that if $R$ is Noetherian and $M$ is in $\sfD^f(R)$, a result similar to Proposition \ref{finite-level} holds with $\sfG$ replaced by $\widetilde{\sfG}$. \end{remark}

\end{subsection}

\begin{subsection}{Ghost maps and the Ghost Lemma}

\begin{definition} Let $\sfT$ be a triangulated category and $\sfC$ a collection of objects from $\sfT$.  A morphism $f:M\to N$ in $\sfT$ is called {\it $\sfC$-ghost} if 
$$\Hom_{\sfT}(\susp^n A, f):\Hom_{\sfT}(\susp^n A, M)\to \Hom_{\sfT}(\susp^n A, N)$$
is zero for every object $A$ of $\sfC$ and all $n$.   Similarly, $f$ is called {\it $\sfC$-coghost} if 
$$\Hom_{\sfT}(f, \susp^n A):\Hom_{\sfT}(N, \susp^n A)\to \Hom_{\sfT}(M, \susp^n A)$$
is zero for every object $A$ of $\sfC$ and all $n$. 
\end{definition}

\begin{remark} \label{ghost-remark1} Suppose $\sfC$ is a collection of objects from $\sfD(R)$. Note that the functor $\Hom_{\sfD(R)}(\susp^n A, -)$ is naturally equivalent to $\Ext^{-n}_{R}(A,-)$.  Thus, a morphism $f:M\to N$ in $\sfD(R)$ is $\sfC$-ghost if and only if for every complex $A$ in $\sfC$ the induced maps $\Ext^n_R(A,M)\to \Ext^{n}_R(A,N)$ are zero for all $n$.  Equivalently, $f:M\to N$ is $\sfC$-ghost if and only if the map $\RHom_R(A,M)\to \RHom_R(A,N)$ induces the zero map on homology for all complexes $A$ in $\sfC$.  Similarly, $f:M\to N$ is $\sfC$-coghost if and only if the map 
$\RHom_R(N,A)\to \RHom_R(M,A)$ induces the zero map on homology for all complexes $A$ in $\sfC$.
\end{remark}

\begin{remark} \label{ghost-remark2} A morphism $f:M\to N$ in $\sfD(R)$ is $R$-ghost  if and only if the induced map $f_*:\hh_*(M)\to \hh_*(N)$ on homology is zero.  Moreover, $f$ is $R$-ghost if and only if $f$ is $\sfP$-ghost (equivalently, $\widetilde{\sfP}$-ghost).  If $R$ is Noetherian and $M$ and $N$ are in $\sfD^f(R)$ then $f$ is $R$-coghost if and only if $f$ is $\sfP$-coghost, since the functors $\RHom_R(M,-)$ and $\RHom_R(N,-)$ commute with (arbitrary) direct sums.

\end{remark}

In general, upper bounds on the level of an object are easier to obtain than lower bounds. For example, Corollaries \ref{level-inequality}, \ref{level-pd}, and \ref{easy-inequality} give upper bounds on levels which follow easily from the definition and elementary properties.   A key tool for obtaining lower bounds is the following result, known as the ``Ghost Lemma''.  It was first proved by G. Kelly in 1965 \cite[Theorem 3]{Ke}. (See also \cite[Lemma 4.11]{R}.)  There is a version for both ghost maps and coghost maps:

\begin{theorem} \label{ghost} Let $\sfT$ be a triangulated category, $\sfC$ a collection of objects of $\sfT$, and $f_i:X_i\to X_{i+1}$ for $0\le i\le n-1$ a sequence of maps in $\sfT$ such that $f_{n-1}f_{n-2} \cdots f_0$ is a nonzero morphism in $\sfT$.  
\begin{enumerate}[(a)]
\item {\rm (Ghost Lemma)} If each $f_i$ is $\sfC$-ghost then $\level_{\sfT}^{\sfC}X_0\ge n+1$.
\item {\rm(Coghost Lemma)} If each $f_i$ is $\sfC$-coghost then $\level_{\sfT}^{\sfC}X_{n}\ge n+1$.
\end{enumerate}
\end{theorem} 

There is an important converse to the Coghost Lemma in the case $\sfT=\sfD^f_b(R)$ proved by Oppermann and \v{S}\v{t}ov\'i\v cek \cite[Theorem 24]{OS}:

\begin{theorem} \label{converse} Suppose $R$ is Noetherian, $M$ and $C$ are objects in $\sfD_b^f(R)$, and that $\level_R^C M\ge n+1$ for some $n\ge 1$.  Then there exist $C$-coghost maps
$f_i:M_i\to M_{i+1}$ for $0\le i\le n-1$ in $D^f_b(R)$ with $M_n=M$ and $f_{n-1}f_{n-2} \cdots f_0$ a nonzero morphism.
\end{theorem}

We note that a converse of the Ghost Lemma for $\sfD^f_b(R)$ has been proved by J. Letz in the case $R$ is a quotient of a Gorenstein ring of finite dimension \cite[2.13]{L}.  However, it is unknown whether such a result holds for all commutative Noetherian rings.

As an application of these results, we prove that for objects in $\sfD_+^f(R)$ the level with respect to $R$ is the same as the level with respect to $\sfP$:

\begin{proposition} \label{P-level}  Let $R$ be Noetherian and $M$ an object in $\sfD_+^f(R)$.  Then $$\level_R^{\sfP} M=\level_R^R M.$$
\end{proposition}
\begin{proof} We first note that if $M$ is not (homologically) bounded above both quantities must be infinite.  Thus, we may assume $M$ is in $\sfD_b^f(R)$.  The inequality $\level_R^{\sfP} M\le \level_R^R M$ is clear.  The reverse inequality is clear if $\level_R^R M\le 1$.  Suppose $\level_R^R M=n\ge 2$.  It suffices to prove $\level_R^{\sfP} M\ge n$.  By the converse coghost lemma,  there exist $R$-coghost maps
$f_i:M_i\to M_{i+1}$ for $0\le i\le n-1$ in $D^f_b(R)$ with $M_n=M$ and $f_{n-1}f_{n-2} \cdots f_0$ a nonzero morphism.  As noted in Remark \ref{ghost-remark2}, the maps $f_i$ are also $\sfP$-coghost.  Since $f_{n-1}f_{n-2} \cdots f_0$ is  a nonzero morphism we obtain that $\level_R^{\sfP} M\ge n$ by the Coghost Lemma.

\end{proof}

\end{subsection}

\section{Main Results}

We begin with a couple of technical results:

\begin{lemma} \label{tricky-lemma}
Consider a diagram of $R$-modules and $R$-linear maps with exact rows and such that the squares commute:

\begin{center} \begin{tikzcd}
0 \arrow[r] & A \arrow[r,"f"] \arrow[d, "\alpha"] & B \arrow[r,"g"] \arrow[d, "\beta"]  \arrow[ld, dashed, "\exists \phi", labels=above left]& C \arrow[r] \arrow[d, "\gamma"] & 0 \\
0 \arrow[r] &D \arrow[r,"i"]                                 & E \arrow[r,"j"]                  & F \arrow[r]                             & 0.
\end{tikzcd}
\end{center}
Suppose that
\begin{enumerate}
\item there exists an $R$-linear map $\phi:B\to D$ such that $\alpha=\phi f$, and 
\item the induced map $\Ext_R^1(F,D)\to \Ext_R^1(C,D)$ is injective.
\end{enumerate}
Then the bottom row splits.
\end{lemma}

\begin{proof}
Applying $\Hom_R(-,D)$ we get a commutative diagram with exact rows:

\begin{center}\begin{tikzcd}
{\Hom_R(E,D)} \arrow[r, "i^*"] \arrow[d]   & {\Hom_R(D,D)} \arrow[r,"\delta" ] \arrow[d, "\alpha^*"] & {\Ext_R^1(F,D)} \arrow[d, hook] \\
{\Hom_R(B,D)} \arrow[r, "f^*"] & {\Hom_R(A,D)} \arrow[r, "\epsilon"]                         & {\Ext^1_R(C,D)}                
\end{tikzcd}
\end{center}

Note that  $\alpha^*(1_D)=\alpha =\phi f=f^*(\phi) \in \im f^*=\ker \epsilon$. Thus $1_D \in \ker\delta=\im i^*$, by the assumed injectivity of the right-most vertical map. That is, $1_D=\sigma i$ for some $\sigma : E\to D$. 
\end{proof}

\begin{lemma} \label{qi} Let $f:P\to Q$  be a quasi-isomorphism of $R$-complexes of Gorenstein projective modules such that $P^{\#}$ and $Q^{\#}$ are bounded below.   Then for any $R$-module $M$ of finite projective dimension and for all integers $i\ge 1$ and all $v$, we have the following isomorphisms induced by $f$:
\begin{enumerate}[(a)]
\item $\Ext^i_R(\cc_v(Q),M)\cong \Ext^i_R(\cc_v(P), M)$;
\item $\Ext^i_R(\bb_v(Q),M)\cong \Ext^i_R(\bb_v(P), M)$.
\end{enumerate}
\end{lemma}
\begin{proof}
Let $n=\min\{\inf P^{\#}, \inf Q^{\#}\}$.  Both isomorphisms clearly hold for all $i$ and $v<n$.  Let $j\ge n$ and assume the isomorphisms hold for all $i\ge 1$ and all $v<j$.   We have the following commutative diagram where the vertical arrows are induced by $f$:

\begin{center}\begin{tikzcd}
0 \arrow[r] &{\hh_j(P)}\arrow[r] \arrow[d]   & {\cc_j(P)} \arrow[r] \arrow[d] & {\bb_{j-1}(P)} \arrow[d] \arrow[r] & 0 \\
0 \arrow[r] & {\hh_j(Q)} \arrow[r] & {\cc_j(Q)} \arrow[r]                        & {\bb_{j-1}(Q)}      \arrow[r] & 0.        
\end{tikzcd}
\end{center}

Since $f$ is a quasi-isomorphism, the left-most vertical arrow is an isomorphism.  From the long exact sequences on $\Ext^i_R(-,M)$ and using the induction hypothesis for $v=j-1$, we see that $\Ext^i_R(\cc_j(Q),M)\cong \Ext^i_R(\cc_j(P), M)$ for all $i\ge 1$ by the Five Lemma.

Consider now the commutative diagram

\begin{center}\begin{tikzcd}
0 \arrow[r] &{\bb_j(P)}\arrow[r] \arrow[d]   & {P_j} \arrow[r] \arrow[d] & {\cc_{j}(P)} \arrow[d] \arrow[r] & 0 \\
0 \arrow[r] & {\bb_j(Q)} \arrow[r] & {Q_j} \arrow[r]                        & {\cc_{j}(Q)}      \arrow[r] & 0          
\end{tikzcd}
\end{center}
where again the vertical maps are induced by $f$.   From the induced long exact sequences on $\Ext^i_R(-,M)$, the isomorphisms $\Ext^i_R(\cc_j(Q),M)\cong \Ext^i_R(\cc_j(P),M)$ for all $i\ge 1$, and  $\Ext^i_R(P_j,M)=\Ext^i_R(Q_j,M)=0$ for all $i\ge 1$ (Lemma \ref{vanishing}), we obtain that $\Ext_R^i(\bb_j(Q),M)\cong \Ext^i_R(\bb_j(P),M)$ for all $i\ge 1$.

\end{proof}

\begin{theorem} \label{main-theorem} Let $M$ be a nonzero object in $\sfD_+(R)$. Then $$\level_R^{\sfG} M\ge \Gpd_R M - \sup M+1.$$
Moreover, if $R$ is Noetherian and $M$ is in $\sfD^f_+(R)$, then
$$\level_R^{\widetilde{\sfG}} M\ge \Gpd_R M - \sup M+1.$$

\end{theorem}
\begin{proof}  We prove the first statement.  The second statement is proved similarly.

 We may assume $\level_R^{\sfG}(M)<\infty$.  Hence $\Gpd_RM<\infty$ by Proposition \ref{finite-level}.  Set $n:=\sup M$, $\ell:=\inf M$ and $g:=\Gpd_R M$.  Certainly $g\ge n\ge \ell$ as $M\not\simeq 0$.  If $g=n$ then the inequality is clear, as $M\not\simeq 0$.  Suppose now that $g>n$.   By Theorem \ref{CI-thm},  there exists an $R$-complex $X$ such that $X\simeq M$ in $\sfD_+(R)$, $X_i=0$ for $i>g$ or $i<\ell$, $X_i$ is projective for all $i\neq n$, and $X_n$ is Gorenstein projective.  Without loss of generality, we may replace $M$ with $X$ in the theorem.  

For any integer $i$ let $\phi_i:X_{\ge i}\to X_{\ge i+1}$ be the natural map of truncated complexes.  

\medskip
\noindent {\it Claim 1:} For all $i\ge n$ we have $\phi_i$ is $\sfG$-ghost.

\smallskip
\noindent {\it Proof of Claim 1:}  By Remark \ref{ghost-remark1}, it suffices to prove that for all $i\ge n$ the induced map $\Ext^j_R(A,X_{\ge i})\to \Ext^j_R(A,X_{\ge i+1})$ is zero for all $j$ and all $A\in \sfG$.  Note that for any $i\ge n$, $X_{\ge i}\simeq \susp^i\, \cc_i(X)$.   Hence, it suffices to prove that for any $i\ge n$ the map $\Ext^j_R(A,\cc_i(X))\to \Ext^{j+1}_R(A, \cc_{i+1}(X))$ is zero for all $j$ and all $A\in \sfG$.  Since $A$ and $\cc_i(X)$ are modules, this is clear for all  $j<0$.  If $j\ge 0$ and $i\ge n$, we note that $\pd_R \cc_{i+1}(X)<\infty$ since $X_k$ is projective for all $k\ge i+1$.
Hence, $\Ext^{j+1}_R(A, \cc_{i+1}(X))=0$ for all $A\in \sfG$ by Lemma \ref{vanishing}. 

\medskip
Let $\rho$ be the natural truncation map $X\to X_{\ge n}$ and $\phi_n'=\phi_n\rho$.  As $\phi_n$ is $\sfG$-ghost so is $\phi_n'$.  Now let $\psi=\phi_{g-1}\phi_{g-2}\cdots\phi_{n+1}\phi_n':X\to X_{\ge g}$.
Then $\psi$ is a composition of $g-n$ $\sfG$-ghost maps.

\medskip
\noindent {\it Claim 2:} $\psi$ induces a nonzero morphism in $\sfD_+(R)$.

\smallskip
\noindent {\it Proof of Claim 2:}   Suppose $\psi=0$ in $\sfD_+(R)$. Choose a semi-projective resolution $\sigma:P\to X$ with $\inf P^{\#}=\inf X^{\#}=\ell$ (cf. \cite[1.7]{AF}).  Then $\psi\sigma:P\to X_{\ge g}$ induces the zero morphism $\sfD_+(R)$.  Since $P$ is semi-projective, this implies $\psi\sigma$ is null-homotopic.    This means there exists a map $\tau: P_{g-1}\to X_g$ such that $\sigma_g=\tau\partial^P_g$ where $\partial^P_g:P_g\to P_{g-1}$ is the $n$th differential of the complex $P$. (Here we are using that $X_i=0$ for $i>g$.)  Hence we obtain the following diagram where the squares commute and $\tau\overline{\partial^P_g}=\overline{\sigma}_g$:

\begin{center}\begin{tikzcd}
0 \arrow[r ] &{\cc_g(P)}\arrow[r, "\overline{\partial^P_g}"] \arrow[d, "\overline{\sigma}_g"]   & {P_{g-1}} \arrow[r] \arrow[d, "\sigma_{g-1}" ] \arrow[ld, dashed, "\tau", labels=above left] & {\cc_{g-1}(P)} \arrow[d, "\overline{\sigma}_{g-1}"] \arrow[r] & 0 \\
0 \arrow[r] & {X_g} \arrow[r, "\partial^X_g", labels=below] & {X_{g-1}} \arrow[r]                        & {\cc_{g-1}(X)}      \arrow[r] & 0          
\end{tikzcd}
\end{center}
Note that both rows are exact, as $g>n=\sup X=\sup P$. Now $\sigma:P\to X$ is a quasi-isomorphism of complexes of Gorenstein projective modules and where $\inf P=\inf X$ is finite.  As $X_g$ is projective  we have by Lemma \ref{qi}(a) that the induced map $\Ext^1_R(\cc_{g-1}(X), X_g)\to \Ext^1_R(\cc_{g-1}(P), X_g)$ is an isomorphism.  Hence, by Lemma \ref{tricky-lemma} we get that the map $\partial_g^X$ splits.  Thus, $\cc_{g-1}(X)$ is isomorphic to a direct summand of $X_{g-1}$, which is Gorenstein projective.  Hence, $\cc_{g-1}(X)$ is Gorenstein projective.  As $g-1\ge n=\sup X$, this implies $g=\Gpd_R X\le g-1$ by Proposition \ref{Gpd-prop}, a contradiction.  

\medskip
Since $\psi:X\to X_{\ge g}$ is a composition of $g-n$ $\sfG$-ghost maps and is nonzero in $\sfD_+(R)$, we have by the Ghost Lemma (Theorem \ref{ghost}) that $\level_R^{\sfG}(X)\ge g-n+1$.  As $\level_R^{\sfG}X=\level_R^{\sfG}M$, this concludes the proof.

\end{proof}

\begin{remark}After a preliminary version of this paper appeared, R. Takahashi pointed out that to the authors that in the case $M$ is a module, an alternative proof of Theorem \ref{main-theorem} can be obtained using \cite[Theorem 1.2]{AT}.
\end{remark}

As an immediate consequence, we have the following generalization of \cite[Proposition 4.5]{JChr} and \cite[Cor 2.2]{AGMSV}:

\begin{corollary} \label{module-case} For a nonzero $R$-module $M$ we have
$$\level_R^{\sfG}M=\Gpd_RM+1.$$  If in addition $R$ is Noetherian and $M$ is finitely generated, we have
$$\level_R^{\widetilde{\sfG}}M=\Gpd_RM+1.$$

\end{corollary}
\begin{proof} These statements follow readily from Theorem \ref{main-theorem} and Corollary \ref{easy-inequality}.
\end{proof}

One unresolved question we have is whether an analogous result to Proposition \ref{P-level} holds for Gorenstein projectives.  

\begin{question} Suppose $R$ is Noetherian and $M$ a complex in $\sfD^f_+(R)$.  Is $\level_R^{\sfG} M=\level_R^{\widetilde{\sfG}} M$?  
\end{question}

We can answer this question affirmatively in the case of modules using Corollary \ref{module-case}:

\begin{proposition} Let $M$ be an $R$-module.
\begin{enumerate}[(a)]
\item If $\level_R^{\sfP}M<\infty$ then $\level_R^{\sfP} M=\level_R^{\sfG} M$.
\item If $R$ is Noetherian and $M$ is finitely generated then $\level_R^{\sfG} M= \level_R^{\widetilde{\sfG}} M.$
\item If $R$ is Noetherian and  $\level_R^RM<\infty$, then $\level_R^R M=\level_R^{\widetilde{\sfG}} M.$
\end{enumerate}
\end{proposition}
\begin{proof}  We may assume $M$ is nonzero.   For the first assertion, we have by \cite[Cor 2.2]{AGMSV} that $\level^{\sfP}_R M=\pd_R M+1$. Thus, $\pd_RM<\infty$.  Then by \cite[Proposition 2.27]{H}, we have $\Gpd_RM=\pd_RM$.  The result now follows from Corollary \ref{module-case}.  The second assertion follows immediately from Corollary \ref{module-case}.  The third statement follows from the first two, along with Proposition \ref{P-level}.
\end{proof}

We obtain the following characterization of Gorenstein local rings:

\begin{corollary} \label{Gor-ring}  Let $R$ be a local Noetherian ring with residue field $k$.  The following conditions are equivalent:
\begin{enumerate}[(a)]
\item $R$ is Gorenstein;
\item $\level_R^{\sfG} k<\infty$;
\item $\level_R^{\sfG} k=\dim R+1$;
\item $\level_R^{\sfG}M\le \dim R +\sup M-\inf M+1$ for all nonzero complexes $M$ in $\sfD_+(R)$.
\item $\level_R^{\widetilde{\sfG}}M\le \dim R +\sup M-\inf M+1$ for all nonzero complexes $M$ in $\sfD^f_+(R)$.
\end{enumerate}
\end{corollary}
\begin{proof}  This follows from Theorem \ref{Gor-classic}, Theorem \ref{main-theorem}, Corollary \ref{easy-inequality} and Corollary \ref{module-case}.
\end{proof}

In \cite{ABIM}, the following upper bound on level with respect to $R$ is proved:

\begin{theorem} \label{ABIM}{\rm (\cite[Theorem 5.5]{ABIM})} Let $R$ be Noetherian and $M$ a nonzero complex in $\sfD^f_b(R)$. Then 
$$\level_R^RM\le \pd_R \hh(M)+1.$$
In particular, if $R$ is regular of finite dimension, then $\level_R^R(M)\le \dim R+1$.
\end{theorem}

One can ask whether either inequality holds if projective dimension is replaced by Gorenstein projective dimension, and level with respect to $R$ replaced by level with respect to $\widetilde{\sfG}$.  The answer is no, as the following example demonstrates:

\begin{example} \label{eg}  Let $k$ be a field and $R=k[x]/(x^2)$ and $F$ the complex $$0\to R\xrightarrow{x} R\to 0,$$ where the modules $R$ sit in homological degrees 1 and 0.
Note that $\hh(F)$ is finitely generated and nonzero.  Since $R$ is a zero-dimensional Gorenstein ring, we have $\Gpd_R \hh(F)=0$.  
We claim that $\level_R^{\widetilde{\sfG}}(F)=2$.   From Corollary \ref{easy-inequality}, we have that $\level_R^{\widetilde{\sfG}}(F)\le \Gpd_RF+1\le 2$.  Suppose  $\level_R^{\widetilde{\sfG}}(F)\le 1$.  Then by Proposition \ref{finite-level}, $F\simeq T$ in $\sfD(R)$, where $T$ is a bounded complex of finitely generated Gorenstein projective modules with zero differentials.  Since $T$ has zero differentials, we have $T\simeq \hh(T)$ in $\sfD(R)$.   Thus, $F\simeq T\simeq \hh(T)\simeq \hh(F)$ in $\sfD(R)$.  Since $F$ is semi-projective, this means there exists a quasi-isomorphism $\sigma:F\to \hh(F)$.   Let $t=\sigma_1(1)\in \hh(F)_1=xR$.   The induced map on homology  $\sigma^*_1:xR\to xR$ is multiplication by $t$, which is the zero map.  This contradicts that $\sigma_1^*$ is an isomorphism.  
Hence, $\level_R^{\widetilde{\sfG}}(F)=2$.

\end{example}

\begin{remark} The argument in Example \ref{eg} applies to any zero-dimensional Gorenstein local ring which is not a field, with $x$ being any nonzero element of the maximal ideal.
\end{remark}

The following result provides a global bound on the levels of complexes with respect to $\sfG$ over a local Gorenstein ring:

\begin{theorem} \label{global-bound} Let $R$ be a  local Gorenstein ring and $M$ a complex in $\sfD_b(R)$.  Then 
$$\level_R^{\sfG} M\le 2(\dim R + 1).$$
Similarly, if $M$ is a complex in $\sfD^f_b(R)$ then $\level_R^{\widetilde{\sfG}} M\le 2(\dim R+1)$.
\end{theorem}
\begin{proof}  Let $Z$ and $B$ be the subcomplexes of $M$ consisting of the cycles and boundaries of $M$, respectively.   
If $M$ is in $\sfD_b^f(R)$, we can assume $Z$ and $B$ are finitely generated by replacing $M$, if necessary, with a semi-projective resolution consisting of finitely generated projective modules in each degree.  As $R$ is Gorenstein, $\Gpd_R L\le \dim R$ for every $R$-module $L$ (\cite[Corollary 2.4]{EJX}).  By Corollary \ref{easy-inequality},  $\level_R^{\sfG} L\le \dim R+1$, and if $L$ is finitely generated, $\level_R^{\widetilde{\sfG}}L\le \dim R+1$. Note that $Z$ and $B$ are bounded complexes with zero differentials.  Since level is invariant under direct sums and suspensions, we see that $\level_R^{\sfG} Z$ and $\level_R^{\sfG} B$ are each bounded above by $\dim R+1$; similarly for $\level_R^{\widetilde{\sfG}} Z$ and $\level_R^{\widetilde{\sfG}} B$ in the case $M$ is in $\sfD_b^f(R)$. As the short exact sequence of complexes
$$0\to Z\to M \to \susp B\to 0$$ induces an exact triangle $Z\to M\to \susp B\to \susp Z$ in $\sfD_b(R)$, the theorem follows by part (5) of Proposition \ref{level-facts}.
\end{proof}

\begin{remark} Example \ref{eg} shows that Theorem \ref{global-bound} is sharp for an arbitrary Gorenstein ring of dimension zero.  We do not know if the bound is the best possible in higher dimensions.
\end{remark}

\end{document}